\documentclass[twoside]{amsart}

\usepackage{epsfig,comment}
\usepackage{mathtools}
\usepackage[subnum]{cases}
\usepackage{enumerate}
\usepackage{bbm,bm}
\usepackage{amsmath,amssymb,mathrsfs}
\usepackage{graphicx} % Required for inserting images
\usepackage{subcaption}
\usepackage{todonotes}

\newtheorem{theorem}{Theorem}[section]

\newtheorem{lemma}[theorem]{Lemma}
\newtheorem{claim}[theorem]{Claim}

\newtheorem{conjecture}[theorem]{Conjecture}

\renewcommand{\Re}{\mathbb{R}}

\DeclareMathOperator{\conv}{conv}

\DeclareMathOperator{\bd}{bd}

\DeclareMathOperator{\inter}{int}

\begin{document}

\title[Covering a square]{Covering a square by congruent squares}

\author[G. D\'osa]{Gy\"orgy D\'osa}
\author[Zs. L\'angi]{Zsolt L\'angi}
\author[Zs. Tuza]{Zsolt Tuza}

\address{Gy\"orgy D\'osa, Department of Mathematics, University of Pannonia, Veszpr\'em, Hungary}
\email{dosa.gyorgy@mik.uni-pannon.hu}
\address{Zsolt L\'angi, Bolyai Institute, University of Szeged, Szeged, Hungary, and\\
Alfr\'ed R\'enyi Institute of Mathematics, Budapest, Hungary}
\email{zlangi@server.math.u-szeged.hu}
\address{Zsolt Tuza, Department of Mathematics, University of Pannonia, Veszpr\'em, Hungary, and\\
HUN-REN Alfr\'ed R\'enyi Institute of Mathematics, Budapest, Hungary}
\email{tuza.zsolt@mik.uni-pannon.hu}

\thanks{Partially supported by the National Research, Development and Innovation Office, NKFI, K-147544 grant, the ERC Advanced Grant “ERMiD”, and the Projects no. TKP2021-NVA-09 and 2024-1.2.8-T\'ET-IPARI-CN-2025-00011, with the support provided by the Ministry of Innovation and Technology of Hungary from the National Research, Development and Innovation Fund.}


\begin{abstract}
The main goal of this paper is to address the following problem: given a positive integer $n$, find the largest value $S(n)$ such that a square of edge length $S(n)$ in the Euclidean plane can be covered by $n$ unit squares. We investigate also the variant in which the goal is to cover only the boundary of a square. We show that these two problems are equivalent for $n \leq 4$, but not for $n=5$. For both problems, we also present
% proofs of 
the solutions for $n=5$.
\end{abstract}

\maketitle

\section{Introduction}

Our main goal is to investigate the following problem from \cite{Friedcov}: given a positive integer $n$, find the
largest value $S(n)$ such that a square of edge length $S(n)$ in the Euclidean
plane can be covered by $n$ unit squares.

The case $n=3$ of this problem was proposed by Dudeney \cite{Dudeney} in 1931, 
%gyuri
and a sketch of its solution, namely that $S(3) = \sqrt{\frac{1+\sqrt{5}}{2}}$, can be found in a paper of Friedman and Paterson \cite{Friedcov}. In their paper, they also give a short argument to show that $S(2)=1$, and remark that a similar method can be applied to prove that $S(5)=2$ and $S(10)=3$.
Complete proofs for the last two equalities were published by Januszewski 
\cite{Janu09} three years later. Some efficient coverings with larger values of $n$, in particular for $n=7$ and $n=8$, are given in \cite{Friedcov}.

%A special case
A weaker version of this problem was proposed by Soifer, see e.g.\ \cite{Soifer06}: for a positive integer $n$, find the smallest number $\Pi(n)$ such that $\Pi(n)$ unit squares can cover a square of edge length $n + \varepsilon$; he gave asymptotic estimates on this number as $n \to \infty$. To our knowledge, the best known bounds in this direction are $\Pi(1)=3$; $6\leq\Pi(2)\leq7$; and $11\leq\Pi(3)\leq 14$, where the inequalities $6\leq\Pi(2)$ and $11\leq\Pi(3)$ follow from the result in \cite{Janu09}. 

A different square covering problem was investigated in \cite{Balogh25}. There, the authors consider squares of sizes
$1,2,\dots,n$, one of each size, and are interested in covering the largest
possible square with translates of these squares. For small values of $n$, this problem is solved by a solver, for larger values of $n$ by a heuristic, and for very large values of $n$ by an asymptotically optimal algorithm.

A higher-dimensional related question was posed by W. Kuperberg \cite{Kuperberg}. He asks
for the smallest number $C(d)$ of unit cubes in the $d$-dimensional Euclidean space that cover a cube of edge length $1+\varepsilon$. For $d=1$, this question is trivial and $C(1)=2$, and for $d=2$ we have $C(2)=\Pi(1) = 3$.
Kuperberg conjectures that $C(d)=d+1$ holds in any dimension.

%Contrary to
As a dual version of the above, there is a rich literature dealing with not covering,
but packing squares into a larger square. One of the first publications
regarding this problem is a paper of Erd\H{o}s and Graham in 1975 \cite{EG75}, who proposed the problem of finding a square with minimum edge length $s(n)$ that can contain a packing of $n$ unit squares for a positive integer $n$.

In fact, Erd\H{o}s considered the problem much earlier (as we are informed by
\cite{Praton}). 
%\H{o}s
He showed that the sum of
circumferences of two non-overlapping squares
(of any size) inside a unit square is at most $4$. This observation was given by Erd\H{o}s around
1932 as a problem for high-school students in Hungary.
This is actually the
simplest case of a more general conjecture in \cite{EG75}, which states that
the total circumference
of any $k^{2}+1$ non-overlapping squares inside a unit square is at most $4k$.

The survey paper \cite{Friedpack} of Friedman, published in 1998, presents many configurations that might be optimal for the problem in \cite{EG75},
and also gives some results. Despite the work of many mathematicians (e.g.\ Stromquist, Cottingham, Wainwright, Stenlund, Trump, etc.), the problem is still open for almost all values of $n$. We recommend the paper \cite{Friedpack} to the interested reader.

Consider $n$ squares of edge lengths $1, 2, \ldots, n$, respectively. The problem of finding the smallest square that can contain a packing of translates of these squares is discussed on the page of A005842 on the website of integer sequences \cite{1}. This question was originated in the
column written by Martin Gardner in 1966 \cite{Gardner66} and popularized further by him also a decade later in \cite{Gardner75}.
%For this problem, 
Based on earlier works, Gardner \cite{Gardner75} listed the best possible (tight)
results up to $n=17$. 
%An upper bound ($UB$) is best possible if it meets a valid lower bound.
Some further research on this problem can be found e.g.\ in the
publications \cite{Korf03, Korf04, Hougardy12}.
The current champion results of upper bounds $UB(n)$ up to $n=56$ can be found on the
\cite{1} website; most of
those values are tight. However, tightness is still not proved for the cases
$n=38,40,42,48,52,53$ and $55$. Hougardy \cite{Hougardy12} showed that
$UB(28)=89$, $UB(32)=108$, $UB(33)=113$, $UB(34)=118$, and $UB(47)=190$.

The paper of Balogh et al.\ \cite{Balogh22} considers this problem where $n$ is large. 
%gyuri
Their result is the first one with asymptotic-type guarantees to the problem of packing consecutive squares into a square. In the paper both
unrestricted axis-parallel packings and guillotine-type packings are
considered. For both cases the paper introduces asymptotically optimal greedy algorithms.

A very unique special case of the problem is $n=24$. Watson mentioned already
in 1918 in \cite{Watson1918} that $n=24$ is the only case (different from
$n=1$) where the sum of the areas of the consecutive squares from $1$ to $n$,
i.e. $\frac{n(n+1)(2n+1)}{6}$ is a perfect square, namely $70\ast70$. However, the
optimum is $71$ as proved by Korf \cite{Korf04} applying a computer-aided
proof. The first purely theoretical proof was recently published by Sgall et
al.\ \cite{Sgall}.

A related question in the literature is the problem of packing an arbitrary
set of squares with a total area of at most $1$ into a rectangle of minimum
area, see \cite{MM67} and \cite{Kleitman75}. Possibly the best result so far
is given in \cite{Hougardy11}, proving that a rectangle with an area smaller
than $1.4$ suffices to encompass all those squares.

A similar question, where the container is a square, was answered in
\cite{Buchwald16}. Here the authors proved that any set of squares can be packed
into the unit square if (i) the total area of the squares is at most $5/9$ and
(ii) none of the squares has edge length greater than $1/2$. They also proved
that $5/9$ is the best possible such value, under the condition (ii).

%\todo[inline]{ZS -- a következő bekezdés 5.\ sorában el lehetne-e kerülni a "completeness ... complete" szóismétlést?}

In this paper, we revisit the problem of determining $S(n)$ for small values of $n$. We also propose a related problem: for any positive integer $n \geq 2$, determine
the largest value $S_{\bd}(n)$ such that the \emph{boundary} of a square of edge length $S_{\bd}(n)$ can be covered with $n$ unit squares. Here, we clearly have $S(n) \leq S_{\bd}(n)$ for all values of $n$. In the paper, for completeness, we also give more detailed proofs of the equalities $S(2)=1$ and $S(3) = \sqrt{\frac{1+\sqrt{5}}{2}}$. Our arguments follow the sketches of the proofs of the same statements in \cite{Friedcov}. We present the proofs to show that $S_{\bd}(n)=S(n)$ holds for $n=2,3$. We observe further that also $S(4)=S_{\bd}(4)=2$ is valid; moreover, we prove that $S(5)=2 < S_{\bd}(5)$. Based on the methods in our proofs, we also determine the value of $S_{\bd}(n)$ for all positive integers $n \equiv 0,1$ mod $4$, including in particular the value of $S_{\bd}(5)$. In addition, we present a proof of the statement $S(5)=2$ different from (and possibly simpler than) the one in \cite{Janu09}.

We raise the following conjecture.

%\todo[inline]{ZS -- az utolsó mondatot új sorba tettem, de visszacsinálhatjátok ha nem tetszik így.
%Másrészt talán "raise" vagy "pose" vagy "propose" szokásosabb kifejezés egy sejtés felvetésére, mint a "make".}

\begin{conjecture}
No square of edge length $x$ with $x > 2$ can be covered by six unit squares. In other words, $S(6)=2$, or equivalently $\Pi(2)=7$.
\end{conjecture}

Our main results are the following. In their formulation, we set $\varphi= \frac{1+\sqrt{5}}{2}$, and recall the well-known fact that $\varphi$ is the golden ratio.

%\todo[inline]{ZS -- nem sikerült megtalálnom, hol mondtuk korábban hogy $\phi$ az aranymetszés aránya. Tehát nem biztos hogy "recall" a megfelelő szó ide.} Zsolti: ugy ertettem, hogy az irodalombol. Picit kiegeszitettem, de ha van jobb javaslat, irjatok at.

\begin{theorem}\label{thm:smallern}
We have $S_{\bd}(2)=S(2)=1$ and $S_{\bd}(3)=S(3)=\sqrt{\varphi}$. Furthermore, if $\mathcal{F}$ is a family of $n$ unit squares covering the boundary of a square $S$ of edge length $S_{\bd}(n)$, then the following holds.
\begin{itemize}
\item If $n=2$, then at least one element of $\mathcal{F}$ is a translate of $S$.

%\todo[inline]{ZS -- javaslat "is a translate of $S$" helyett: 

%is identical to $S$.} Zsolti: Azt hiszem, ugy nem igaz az allitas, ugyanis ha mindket fedo negyzet eltolt, akkor le lehet fedni ugy, hogy S ket peldanyat oldalakkal parhuzamosan ellentetes iranyban picit eltoljuk.

\item If $n=3$, then the elements of $\mathcal{F}$, together with $S$, form a configuration congruent to the one in Figure~\ref{fig:3squares}. 
\end{itemize}
\end{theorem}

\begin{figure}[ht]
\begin{center}
\includegraphics[width=0.4\textwidth]{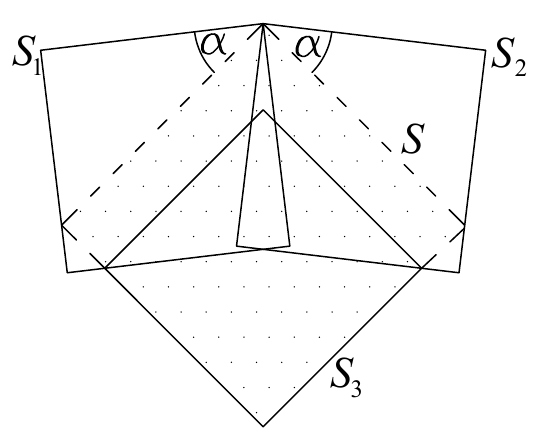}
\caption{Three unit squares $S_1, S_2, S_3$ covering the boundary of a square (and also the whole square) $S$ of edge length $\sqrt{\varphi}$. The square 
%gyuri
$S$ is denoted as a region with dashed boundary; $S_3$ is a homothetic copy of $S$ with a vertex of $S$ as homothety center. The squares $S_1$ and $S_2$ are rotated copies of $S_3$ by angles $\pm \alpha$, where $\alpha = \arccos \frac{1}{\sqrt{\varphi}}$.}
  \label{fig:3squares}
\end{center}
\end{figure}

\begin{theorem}\label{thm:biggern}
We have $S(5)=2$. Furthermore, if $\mathcal{F}$ is a family of $5$ unit squares, covering a square $S$ of edge length $2$, then at least one element of $\mathcal{F}$ has sides parallel to those of $S$. 
\end{theorem}
%Biralo 2

\begin{theorem}\label{thm:bd}
The following exact values are valid.
\begin{enumerate}
\item[$(i)$]
$S_{\bd}(n+4)=S_{\bd}(n)+\sqrt{2}$ for every $n \geq 4$.
\item[$(ii)$] $S_{\bd}(4)=2$, and if $\mathcal{F}$ is a family of $4$ unit squares covering 
the boundary of a 
square $S$ of edge length $2$, then every element of $\mathcal{F}$ is homothetic to $S$.
\item[$(iii)$] $S_{\bd}(5)= \frac{1}{2}\sqrt{2}\sqrt{\sqrt{13-8\sqrt{2}}+1} + 1 \approx 2.072$. Furthermore, if $\mathcal{F}$ is a family of $5$ unit squares, covering
the boundary of a 
square $S$ of edge length $S_{\bd}(5)$, then the configuration is congruent to the one in Figure~\ref{fig:5bdoptimal}.
\end{enumerate}
\end{theorem}

\begin{figure}[ht]
\begin{center}
\includegraphics[width=0.52\textwidth]{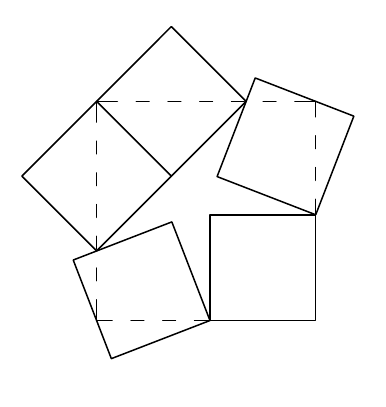}
\caption{Five unit squares covering the boundary of a square $S$ of edge length $\frac{1}{2}\sqrt{2}\sqrt{\sqrt{13-8\sqrt{2}}+1}+1\approx 2.072$. The boundary of $S$ not belonging to the edges of the covering unit squares is denoted by a dashed line.}
  \label{fig:5bdoptimal}
\end{center}
\end{figure}

In this paper we identify points with their position vectors. The Euclidean norm of a point $p \in \Re^2$ is denoted as $||p||$. In particular, the Euclidean distance of $p,q \in \Re^2$ is $||p-q||$. We use the notation $\inter(\cdot), \bd(\cdot)$ and $\conv(\cdot)$ for the \emph{interior}, \emph{boundary} and \emph{convex hull} of a set.

%gyuri %zsolti
The structure of the paper is as follows. %In Sections~\ref{sec:smallern} and \ref{sec:biggern} we prove Theorems~\ref{thm:smallern} and \ref{thm:biggern}, respectively. Finally, in Section~\ref{sec:bd} we prove Theorem~\ref{thm:bd}.
In Sections~\ref{sec:smallern}, \ref{sec:biggern}, and \ref{sec:bd}, we prove Theorems~\ref{thm:smallern},  \ref{thm:biggern}, and \ref{thm:bd}, respectively.

%\todo[inline]{ZS -- nekem kicsit furán hangzik ez az utolsó két sor. Két módosítási lehetőség:

%(a)

%In Sections~\ref{sec:smallern}, \ref{sec:biggern}, and \ref{sec:bd}, we prove Theorems~\ref{thm:smallern},  \ref{thm:biggern}, and \ref{thm:bd}, respectively.

%(b)

%We prove Theorem 1.$i$ in Section $i$ for $i=2,3,4$.}

\section{Proof of Theorem~\ref{thm:smallern}}\label{sec:smallern}

%Gyuri, az Observation beszurasa
We start with an observation that we will need several times.

\begin{claim}\label{obs1}
%Suppose that we have a triangle, the length of its one side is unit and the corresponding height is also one unit, and it is within a unit square. Then one side of the triangle is on one side of the unit square.
Suppose that one side of a triangle has unit length, its corresponding height is also one unit, and the triangle is inside a unit square. Then the unit side of the triangle coincides with a side of the square.
\end{claim}

%\todo[inline]{ZS -- javaslat Claim \ref{obs1} szövegének módosításáta:

%Suppose that one side of a triangle has unit length, its corresponding height is also one unit, and the triangle is inside a unit square. Then the unit side of the triangle coincides with a side of the square.}

\subsection{The proof of the case $n=2$}\label{subsec:nis2}
%Gyuri: ide beteszem a kis bizonyitast
%Biralo 3
Let $S$ be a square of size $\lambda \geq 1$, and assume that there are unit squares $S_1$ and $S_2$ such that $\bd(S) \subseteq S_1 \cup S_2$. We also assume that neither $S_1$ nor $S_2$ is a translate of $S$.
Let the vertices of $S$ be $v_1,v_2, v_3, v_4$ in counterclockwise order. Let $E_i = [v_i,v_{i+1}]$ (indices are taken modulo $4$). Note that as the diameters of $S_1$ and $S_2$ are $\sqrt{2}$, opposite vertices of $S$ belong to different $S_i$. Thus, without loss of generality, we may assume that $v_1, v_2 \in S_1$ and $v_3, v_4 \in S_2$. Then we have $E_1 \subseteq S_1$ and $E_3 \subseteq S_2$, like in Figure ~\ref{fig:3squaresnotation}.
Moreover, by Claim \ref{obs1}, $S_1$ does not intersect $E_3$ and  $S_2$ does not intersect $E_1$, implying that $S_i \cap \bd(S)$ is connected for $i=1,2$, and also that $S \subseteq S_1 \cup S_2$.

From now on, it suffices to prove that $S(2)=1.$

%Let the vertices of $S$ be $v_1,v_2, v_3, v_4$ in counterclockwise order. 
%Let $E_i = [v_i,v_{i+1}]$ (indices are taken modulo $4$). 
% By our assumptions, there is a side of $S$ which belongs entirely to $S_1$, and the % opposite side belongs entirely to $S_2$. 
% Without loss of generality, let us assume that $E_1 \subset S_1$.
%Then we may assume that $v_1, v_2$ lie in consecutive sides of $\bd(S_1)$, like in Figure ~\ref{fig:3squaresnotation}.

\begin{figure}[ht]
\begin{center}
\includegraphics[width=0.5\textwidth]{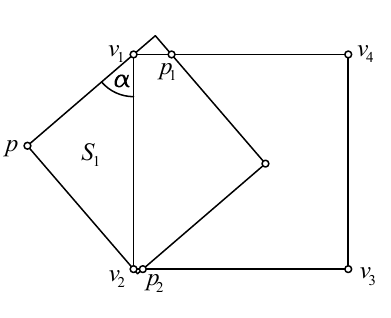}
\caption{Notation for the proof.}
  \label{fig:3squaresnotation}
\end{center}
\end{figure}

Let $p$ denote the vertex of $S_1$ separated from $S$ by the sideline through $E_1$. Let $p_1,p_2$ denote the intersection points of $\bd(S_1)$ with $E_4$ and $E_2$, respectively. Let $\alpha$ denote the angle of the triangle $\conv \{ p,v_1,v_2 \}$ at $v_1$, where, due to our conditions, $0 < \alpha < \frac{\pi}{2}$. Let $A_1(\alpha)$ and $A_2(\alpha)$ denote the lengths of $[p_1,v_1]$ and $[p_2,v_2]$, respectively. An elementary computation shows that
\begin{equation}\label{eq:A1A2}
A_1(\alpha) = \frac{1-\lambda \cos \alpha}{\sin \alpha}, \quad A_2(\alpha) = \frac{1-\lambda \sin \alpha}{\cos \alpha}.
\end{equation}
Then
%gyuri %zsolti
%\begin{equation}
%\lambda - A_1(\alpha)-A_2(\alpha)=\frac{\lambda(1+\cos \alpha \sin \alpha)}
%{\cos \alpha \sin \alpha} \geq \frac{(1-\sin \alpha)(1-\cos \alpha)}{\sin 
%\alpha \cos \alpha} > 0.
\begin{multline}\label{eq:mainforn2}
\lambda-A_{1}(\alpha)-A_{2}(\alpha) =\frac{\lambda(1+\cos\alpha\sin\alpha)-\cos\alpha-\sin\alpha}{\cos
\alpha\sin\alpha}\\
\geq \frac{\left(  1-\cos\alpha\right)  \left(  1-\sin\alpha\right)  }%
{\cos\alpha\sin\alpha}>0 .
\end{multline}
%\end{equation}
Thus, the length of the arc $S_1 \cap \bd(S)$ is strictly less than $2\lambda$. Since the same argument can be applied for $S_2 \cap \bd(S)$, and the perimeter of $S$ is $4\lambda$, we reach a contradiction.

\subsection{The proof of the case $n=3$}\label{subsec:nis3}

To handle this case we follow the same proof in a more elaborate form.
To do it, if $X \subset \Re^2$ is the union of two orthogonal closed segments, meeting at a common endpoint, then we say that $X$ is an \emph{$L$-shape}, and the two segments are called its \emph{legs}. The common endpoint of the legs is called the \emph{corner} of $X$, and the other two endpoints of the legs are called the \emph{endpoints of $X$}.

Let $0 < x < \sqrt{2}$. Let $l(x)$ denote the largest real number $y$ such that an $L$-shape $X$ with legs of lengths $x$ and $y$ is contained in a
% suitable 
unit square. If an $L$-shape $X$ with legs of lengths $x$ and $l(x)$ is contained in a unit square $S$, we say that $X$ is \emph{maximally embedded} in $S$.

\begin{lemma}\label{lem:mainfor3}
Let $X$ be an $L$-shape maximally embedded in a unit square $S$. Assume that the legs of $X$ are of lengths $x$ and $l(x)$, where $x > 1$. Then $l(x) < 1$, the endpoint $p$ of $X$ lying on the longer leg is a vertex of $S$, and the other endpoint and the corner of $X$ are points of the two sides of $S$ not containing $p$.
Furthermore, $l(x)=x-x\sqrt{x^2-1}$.
\end{lemma}

\begin{proof}
The statement that $l(x) < 1$ is trivial. Let $q$ be the endpoint of $X$ on the leg of length $l(x)$, and $c$ be the corner of $X$.
Without loss of generality, we may assume that $p,q,c$ belong to the boundary of $S$. Since $x > 1$, it follows that $c$ is an interior point of a side $E_2$ of $S$. Thus, we also have that $q$ belongs to the interior of a side $E_1$ consecutive to $E_2$. Let $E_3$ be the other consecutive side of $E_2$. If $p$ is in the interior of the fourth side $E_4$, then rotating $X$ around $c$ we find an $L$-shape in $S$ whose shorter leg is longer than $l(x)$, a contradiction.

Assume that $p$ is in the interior of $E_3$. If, in addition, $p$ is not closer to $E_2$ than $q$ is, then we may rotate $X$ around $q$ and find an $L$-shape in $S$ with a longer short leg. 
%Biralo 5
%On the other hand, if $p$ is closer to $E_2$ than $q$ is, then we may rotate $X$ around %the midpoint of the segment $[p,q]$ and find a larger $L$-shape. 
On the other hand, if $p$ is closer to $E_2$ than $q$ is, then we may rotate $X$ around $p$ and find a larger $L$-shape. 
Thus, by the definition of $l(x)$, $p$ is the common point of $E_3$ and $E_4$, proving the first statement. Then the formula for $l(x)$ follows by the Pythagorean theorem and comparing the sides of two similar triangles.
\end{proof}

%Biralo 6,7,8,9, ez mint a 6. oldalon levo bizonyitasra vonatkozik, vagyis a Theorem 1.2.
%   tetel bizonyitasa az n=3 esetben, ennek
% reszbenatirtam a bizonyitast 

%zsolti
%Similarly like in Subsection~\ref{subsec:nis2},

In the proof we also need the following claim, which can be proved by elementary calculus and computation. We only sketch its proof, and we leave the exact computations to the reader.
%\todo[inline]{ZS -- javaslat a fenti mondat végére: computation, which we leave to the reader.}

\begin{claim}\label{obs2}
Let $L_1$ and $L_2$ be two parallel lines at distance $\sqrt{\varphi}$. Let $S$ be a unit square intersecting both $L_1$ and $L_2$. Then the total length of $S \cap (L_1 \cup L_2)$ is at most $2+\sqrt{2}-\sqrt{2+2\sqrt{5}} \approx 0.8701742630$.
\end{claim}

To prove Claim~\ref{obs2}, observe that if we move $S$ perpendicularly to the two lines, the total length of the intersection does not change. Thus, without loss of generality, we may assume that a vertex of $S$ lies on $L_1$. Then we can compute the length $\ell(\alpha)$ of $S \cap L_2$ as a function of the angle $\alpha$ of a suitably chosen diagonal of $S$ and $L_2$, where we may assume that $- \arccos \sqrt{\frac{\varphi}{2}} \leq \alpha \leq  \arccos \sqrt{\frac{\varphi}{2}}$. A numeric computation shows that $\ell(\alpha)$ is not maximal if $|\alpha|$ is `large', and in the opposite case it is a concave function of $\alpha$, implying, combined with the observation that $\ell(\alpha)$ is an even function, that it is maximal at $\alpha = 0$.
%\todo[inline]{ZS -- javaslat: "Let" helyett "To prove the theorem for $n=3$, let"} Zsolti: picit mas szoveggel irtam

Now we prove Theorem~\ref{thm:smallern} for $n=3$. Let $S$ be a square of size $\sqrt{\varphi}$, and assume that there are unit squares $S_1,S_2,S_3$ such that $\bd(S) \subseteq S_1 \cup S_2 \cup S_3$. 
Let the vertices of $S$ be $v_j$, $j=1,2,3,4$, in counterclockwise order. Let $E_i = [v_i,v_{i+1}]$. 
%By our assumptions, there is a side of $S$ which belongs entirely to one of the unit 
%squares, say we assume that $E_1 \subset S_1$.
Note that as the diameters of the $S_i$ are $\sqrt{2}$, no $S_i$ covers opposite vertices of $S$. On the other hand, by the pigeonhole principle, there is some $S_i$ that covers at least two vertices of $S$. Thus, without loss of generality, we may assume that $v_1, v_2 \in S_1$, and $v_3, v_4 \in S_2 \cup S_3$.
Then, by convexity, $E_1 \subset S_1$.

First, consider the case that $v_3, v_4$ are covered by the same square, say $v_3, v_4 \in S_2$. We estimate how long parts of $E_2 \cup E_4$ are covered by $S_1$, $S_2$ and $S_3$. To estimate the length $S_1 \cap (E_2 \cup E_4)$, we can use the formulas from (\ref{eq:A1A2}) with $\lambda = \sqrt{\varphi}$, and obtain, by elementary calculus, that it is at most
\begin{multline*}
\max \left\{ \frac{1-\sqrt{\varphi} \cos \alpha}{\sin \alpha} + \frac{1-\sqrt{\varphi} \sin \alpha}{\cos \alpha} : \arccos \frac{1}{\sqrt{\varphi}} \leq \alpha \leq \arcsin \frac{1}{\sqrt{\varphi}} \right\} = \\ = 2\sqrt{2}-2\sqrt{\varphi} \approx 0.2843878263.
\end{multline*}
We can use the same quantity to estimate the length of $S_2 \cap (E_2 \cup E_4)$. Note that these estimates also imply that $S_1 \cup S_2$ covers neither $E_2$ nor $E_4$. Thus, to estimate the length of $S_3 \cap (E_2 \cap E_4)$, we can apply Claim~\ref{obs2}, leading to a contradiction, as the total length of $E_2 \cup E_4$ is $2\sqrt{\varphi} \approx 2.544039299$.

Next, consider the case that $v_3$ and $v_4$ are covered by different squares, say $v_3 \in S_3$ and $v_4 \in S_2$.  If $S_2$ intersects $E_2$ and $S_3$ intersects $E_4$, then we may use an argument similar to the previous one, applying Claim~\ref{obs2} to estimate the lengths of $S_2 \cap (E_2 \cup E_4)$ and $S_3 \cap (E_2 \cup E_4)$. Thus, without loss of generality, we may assume that the only sides of $S$ that $S_2$ intersects are $E_3$ and $E_4$. But then $E_2$ is covered by $S_1 \cup S_3$. If $S_3$ intersects $E_4$, then the length of $S_3 \cap E_2$ is at most $0.8701742630$ and the length of $S_1 \cap E_2$ is at most $0.2843878263$; a contradiction. Hence, it follows that the only sides of $S$ that $S_4$ intersects are $E_2$ and $E_3$. But then, by convexity, each $S_i$ intersects $\bd(S)$ in a connected polygonal arc. In particular, 

the length of $E_3 \cap S_2$ is at most $l \left( \sqrt{\varphi} - A_1(\alpha) \right)$ and that of $E_3 \cap S_3$ is at most $l \left( \sqrt{\varphi} - A_2(\alpha) \right)$.
Hence, it suffices to show that the next inequality holds, and equality holds only if $\alpha=\arccos \frac{1}{\sqrt{\varphi}}$:
\[
l \left( \sqrt{\varphi} - A_1(\alpha) \right) + l \left( \sqrt{\varphi} - A_2(\alpha) \right) \leq \sqrt{\varphi}.
\]
Let $f(\alpha)$ denote the expression on the right-hand side of the above inequality. Note that $f(\alpha)= f \left( \frac{\pi}{2}- \alpha\right)$ for all $\alpha$, and it is sufficient to investigate this function on the interval $\left[ \arccos \frac{1}{\sqrt{\varphi}}, \frac{\pi}{4} \right]$. On the other hand, we have that
$f\left(\arccos \frac{1}{\sqrt{\varphi}} \right) > f(\alpha)$ if $\frac{7\pi}{32} \leq \alpha \leq \frac{\pi}{4}$, and $f'(\alpha) < 0$ if $\arccos \frac{1}{\sqrt{\varphi}} \leq \alpha \leq \frac{7\pi}{32}$ (see Figure~\ref{fig:Maple}), showing that $f(\alpha) < f\left(\arccos \frac{1}{\sqrt{\varphi}} \right)$ for any $\alpha > \arccos \frac{1}{\sqrt{\varphi}}$. Thus, the assertion follows from the observation that $f\left(\arccos \frac{1}{\sqrt{\varphi}} \right) = \sqrt{\varphi}$.

\begin{figure}
\centering
\begin{minipage}{.5\textwidth}
  \centering
  \includegraphics[width=.9\textwidth]{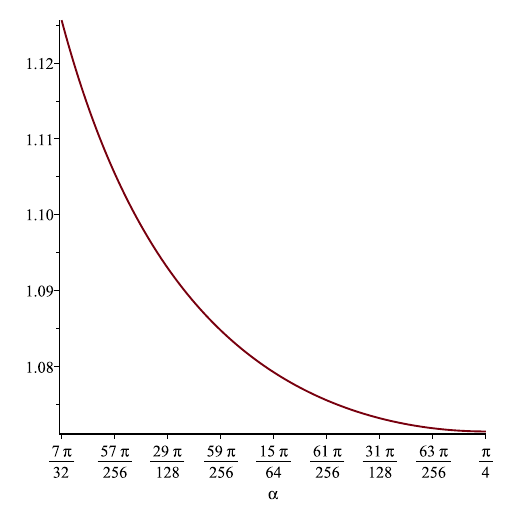}
  %\label{fig:Maple1}
\end{minipage}%
\begin{minipage}{.5\textwidth}
  \centering
  \includegraphics[width=.9\textwidth]{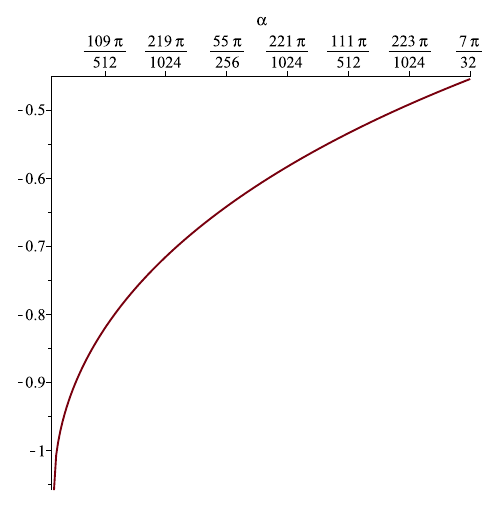}
  %\label{fig:Maple2}
\end{minipage}
\caption{
%zsolti
Left-hand side panel: the graph of $f(\alpha)$. Right-hand side panel: The graph of the numerator of $f'(\alpha)$; the denominator of $f'(\alpha)$ is $\sqrt{ \left( \sqrt{\varphi} -A_1(\alpha) \right)^2-1 }$, which is positive on the investigated interval. 
%gyuri
Computation shows that the numerator of the function $f'$ can be extended to a continuous function on the given interval.}
\label{fig:Maple}
\end{figure}

\section{Proof of Theorem~\ref{thm:biggern}}\label{sec:biggern}

Consider a family $\mathcal{F} = \{ S_i:i=1,2,\ldots,5\}$ of unit squares that cover a square $S$ of edge length $a \geq 2$. For brevity, we call the squares $S_i$ \emph{items}. To prove the theorem by contradiction, assume that either $a > 2$, or $a=2$ and no item has sides parallel to those of $S$.

We denote the vertices of $S$ as $v_1, v_2, v_3, v_4$ in counterclockwise order. For $i=1,2,3,4$ we denote the midpoint of the side $[v_i,v_{i+1}]$ by $m_i$, where the indices are taken modulo $4$. We denote the center of $S$ by $c$. We call the points $v_i, m_j$ and $c$ \emph{important points}. Our investigation is based on analyzing which important point belongs to which item.

%\begin{figure}[ht]
%\begin{center}
%\includegraphics[width=0.6\textwidth]{square1.png}
%\caption{The covered square (i.e. the board), the important lines and
%important points}
%\label{board}
%\end{center}
%\end{figure}

First, note that if $a > 2$, no item can cover two points from amongst the vertices and center of $S$. The same statement also holds for $a=2$, provided that no item has sides parallel to those of $S$. Thus, every item covers exactly one of these five points. We call an item covering a vertex a \emph{vertex item}, and the one covering the center $c$ of $S$ the \emph{central item}. In addition, we observe that since the diameter of the items is $\sqrt{2}$, any item covers at most two (and at least one) of the important points. We label the items in the way to ensure that $c \in S_5$ and $v_i \in S_i$ for $i=1,2,3,4$.

We distinguish three cases. In the proof we often use the function $l(x)$ determined in Lemma~\ref{lem:mainfor3}. We observe that this function strictly decreases on the interval $x \in [1,\sqrt{2}]$, with $l(1)=1$ and $l(\sqrt{2})=0$. Furthermore, the function $x+l(x)$ also strictly decreases on this interval.

\smallskip

\textbf{Case 1:}
\emph{$S_5$ intersects $\bd(S)$ in at most one point.}

In this case $\bd(S)$ is covered by the four vertex items. We use Lemma~\ref{lem:mainfor3}, and conclude that each vertex item
%gyuri
covers an $L$-shape of length
at most $x+l(x)$ for some $1 \leq x \leq \sqrt{2}$, with $l(x)=x-x\sqrt{x^2-1}$. Since $\sqrt{2} \leq x+ l(x) \leq 2$ holds on this interval, with equality on the right if and only if the item is a homothetic copy of $S$, the assertion readily follows.

\smallskip

\textbf{Case 2:}
%gyuri
\emph{$S_5$ contains two important points.}

Without loss of generality, we may assume that $m_4 \in S_5$.
Then the point $m_2$ belongs to $S_2$ or $S_3$. Without loss of generality, we may assume that $m_2 \in S_3$. Then $m_3 \notin S_3$, implying that $m_3 \in S_4$
%zsolti
(see Figure~\ref{fig:cases23}a).

We observe that the segment $[c,m_3]$ is covered by $S_4 \cup S_5$. Indeed, since $v_1 \in S_1$ and $v_2 \in S_2$, no interior point of $[c,m_3]$ belongs to $S_1 \cup S_2$. On the other hand, if $S_3$ intersects $[c,m_3]$ in a nondegenerate segment, then, since $[m_2,v_3 ] \subset S_3$, $S_3$ contains a triangle with edge length $1$ and corresponding height $1+\varepsilon$ for some $\varepsilon > 0$. Since no parallelogram of unit area contains a triangle of area strictly greater than
%gyuri
$1/2$, this is a contradiction, showing that $[c,m_3] \subset S_4 \cup S_5$. In addition, by convexity, $[v_4,m_4] \not\subset S_4 \cup S_5$ implies that $[v_4,m_4] \subset S_1 \cup S_4$, which contradicts the argument in Case 1. Thus, we have that the boundary of the square $\conv \{ v_4,m_4,c,m_3 \}$ is covered by $S_4 \cup S_5$, and the assertion follows from Theorem~\ref{thm:smallern}.

\begin{figure}[ht]
\begin{center}
\includegraphics[width=0.9\textwidth]{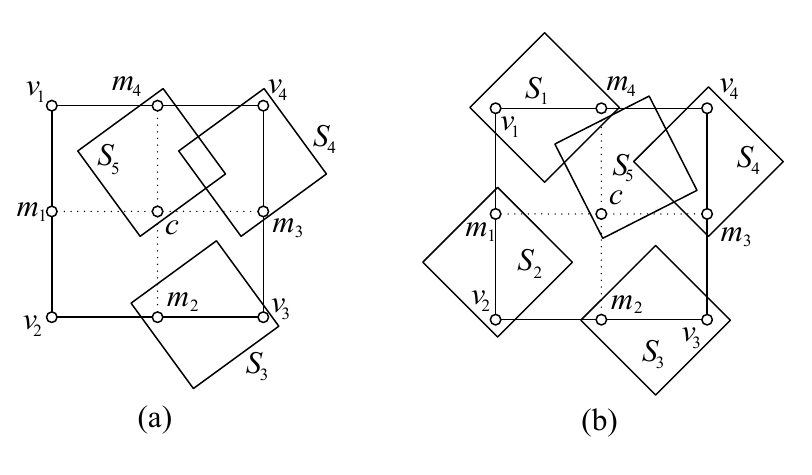}
%\caption{Illustrations for the proof of Theorem~\ref{thm:biggern} in Cases 2 and 3. Panel (a): an illustration for Case 2. Panel (b): an illustration for Case 3.}
\caption{Illustrations for the proof of Theorem~\ref{thm:biggern} in Cases 2 and 3. Panel (a): Case 2, Panel (b): Case 3.}
\label{fig:cases23}
\end{center}
\end{figure}

%\todo{az ábra korábbi aláírása még a fájlban van, könnyen vissza\-rakható ha mégis az a jobb.}

\smallskip

\textbf{Case 3}: \emph{$S_5$ intersects
%gyuri
$\bd(S)$ in a nondegenerate segment, but it contains only one important point, namely $c$.}

Since no $S_i$ can contain more than two important points, we may assume that each $S_i$, apart from $S_5$, contains exactly two of them. More specifically, without loss of generality, we may assume that $m_4 \in S_1$, $m_1 \in S_2$, $m_2 \in S_3$ and $m_3 \in S_4$ (see Figure~\ref{fig:cases23}b). If $S_5$ does not intersect $\bd(S) \setminus \bigcup_{i=1}^4 S_i$, then $\bd(S)$ is covered by the four vertex items, and we may apply the argument in Case 1. Thus, without loss of generality, we may assume that $S_5 \cap [v_4,m_4] \neq \emptyset$, implying that $S_5 \cap [c,m_2]$ has length strictly less than $\sqrt{2}-a/2$. Similarly, $S_4 \cap [v_3,m_3]$ has length at most $\sqrt{2}-a/2$. Since $[m_1,v_2] \subset S_2$, it follows that the remaining part of $[c,m_2] \cup [v_3,m_3]$ is covered by $S_3$. In particular, we have that the length of $S_3 \cap \bd(\conv \{ m_2,v_3,m_3,c \})$ is strictly greater than $3a/2-2(\sqrt{2}-a/2)=5a/2-2\sqrt{2} > a$. This contradicts the observation after (\ref{eq:mainforn2}), stating that if $S'$ is a square of edge length $\lambda > 1$, then no unit square covers an arc of $\bd(S')$ of length greater than
%gyuri
$2\lambda$, and the same statement holds if $\lambda=1$ and the covering square is not a translate of $S'$.

\section{Proof of Theorem~\ref{thm:bd}}\label{sec:bd}

Consider a family $\mathcal{F}= \{ S_1, \ldots, S_n \}$ of $n$ unit squares, with $n \geq 4$, that cover the boundary of a square $S$ of edge length $a > \sqrt{2}$. Then no element of $\mathcal{F}$ covers points from opposite sides of $S$.

First, we prove $(i)$.
Assume that a square $S_i$ covers parts of two consecutive sides of $S$, but not the common vertex $v$ of the sides. Let $q_1,q_2$ be the farthest points of $S_i$ from $v$ on these two sides.  Then $||q_i - v || < \sqrt{2}$ for $i=1,2$. Let the connected component of $\bd(S) \setminus S_i$ containing $v$ be $\Gamma_1$, and the other one be $\Gamma_2$. Since $\bd(S) \subset \bigcup \mathcal{F}$, there is another square $S_j$ that contains $v$. If $S_j$ intersects $\Gamma_2$ at a point $p$, say on the side containing $q_1$, then $[p,v] \subseteq S_j$. Thus, $S_i$ can be moved to cover the segment $[v,q_2]$. If $S_j$ does not intersect $\Gamma_2$, then $S_i$ and $S_j$ can be moved to cover $[v,q_1]$ and $[v,q_2]$, respectively. Thus, in both cases the covering can be modified so that the intersections of $S_i$ and $S_j$ with $\bd(S)$ are connected. Applying this process repeatedly, we may assume that each element of $\mathcal{F}$ intersects $\bd(S)$ in a connected arc. Namely, each of those intersections with $\bd(S)$ is a segment or an $L$-shape (for its definition, see Subsection~\ref{subsec:nis2}). If a square intersects in a segment, we may assume that the length of the segment is $\sqrt{2}$. Furthermore, by Lemma~\ref{lem:mainfor3}, if it is an $L$-shape, we may assume that its legs are of length $x$ and
%gyuri
$l(x)=x-x\sqrt{x^2-1}$ for some $1 \leq x \leq \sqrt{2}$. 

Let the vertices of $S$ be
%gyuri
$v_1, v_2, v_3, v_4$ in counterclockwise order on its boundary.
Note that $l(x)=x -x \sqrt{x^2-1}$ is a strictly decreasing bijection between the intervals $[1,\sqrt{2}]$ and $[0,1]$. Thus, if for some squares $S_i \cap \bd(S)$ and $S_j \cap \bd(S)$ overlap, we may move the elements of $\mathcal{F}$ slightly so that the interiors of the elements of the modified family already cover $S$. Hence, in this case, by slightly enlarging the configuration, we obtain a larger square whose boundary is covered with $n$ unit squares. Thus, in the following we assume that the parts of $\bd(S)$ covered by the $S_i$ are mutually non-overlapping polygonal curves.
%\todo{biztos hogy a nem definiált "arcs" jó szó?}
In particular, for each vertex $v_1, v_2, v_3, v_4$ of $S$, we may assign %gyuri
a square $S_i$, respectively, covering $v_i$ so that it intersects $\bd(S)$ in segments of lengths $x_i$ and $l(x_i)$, and every other element of $S$ intersects a side of $S$ in a segment of length $\sqrt{2}$. We call $S_1, \ldots, S_4$ \emph{vertex squares}, and the other squares \emph{side squares}. Note that it may happen that some $v_i$ belongs to more than one square: it can happen when $x_i=\sqrt{2}$ for some $i$, in which case $v_i$ belongs to a further element of $\mathcal{F}$.

We show that the vertex squares can be chosen so that on each side of $S$ there are $k$ or $k+1$ side squares for some nonnegative integer $k$. Indeed, let there be $k_i$ side squares on the side 
%gyuri
$[v_i,v_{i+1}]$ and $k_j$ side squares on the side 
%gyuri
$[v_j,v_{j+1}]$,
%\todo{subscript i changed to j}
with $k_j \leq k_i$. Then the length of the sides is $\sqrt{2} k_i+s_1+s_2=\sqrt{2} k_j+s_3+s_4$ for some $s_1,s_2,s_3,s_4 \in [0,\sqrt{2}]$. Thus, $k_i-k_j \leq 2$, with equality if and only if $s_1=s_2=0$ and $s_3=s_4=\sqrt{2}$. In particular, if $k_i-k_j=2$, then the edge length of $S$ is an integer multiple of $\sqrt{2}$. If these sides of $S$ are opposite sides ($j=i+2$ mod $4$), then each $v_i$ belongs to more than one element of $\mathcal{F}$, and each side consists of some segments of length $\sqrt{2}$, in which case our statement is trivial by the equality of the lengths of the sides of $S$. If the above two sides of $S$ are consecutive, then, by our conditions, the length of the sides of $S$ are integer multiples of $\sqrt{2}$, which yields that also the fourth vertex of $S$ belongs to two elements of $\mathcal{F}$, and our statement follows.

From now on we assume that for 
%gyuri
$i=1,2,3,4$ the side $E_i=[v_i,v_{i+1}]$ is covered by $k$ or $k+1$ segments of length $\sqrt{2}$, and two legs of $L$-shapes, whose lengths are denoted by $y_i$ and $z_i$ such that the one with length $y_i$ contains $v_i$, and the other one contains $v_{i+1}$, implying also that $\{ z_{i-1}, y_i \} = \{ x, l(x) \}$ for some $x \in [1,\sqrt{2}]$. From this, it also follows that for any $n \geq 4$, we have $S_{\bd}(n+4) = S_{\bd}(n)+\sqrt{2}$.

To prove $(ii)$, observe that $x+l(x) \leq 2$ holds for any $1 \leq x \leq \sqrt{2}$. Thus, the perimeter of $S$ is at most $8+(n-4)\sqrt{2}$. On the other hand, if $n$ is divisible by $4$, this value can be attained if each vertex of $S$ is the vertex of exactly one of the $S_i$, and the remaining parts of the sides are covered by diagonally placed unit squares.

Now we prove $(iii)$. In the proof we set $n=5$, and use the notation of the proof of $(i)$.

Without loss of generality, we assume that the only side square intersects $[v_1,v_2]$ in a segment of length $\sqrt{2}$. We distinguish three cases, depending on whether $y_1$ and $z_1$ are equal to $x$ or $l(x)$ for some $x \in [1,\sqrt{2}]$. We assume that the sidelength of $S$ is at least $\frac{1}{2}\sqrt{2}\sqrt{\sqrt{13-8\sqrt{2}}+1} + 1$.

\smallskip

\textbf{Case 1}:
\emph{$y_1=x_1$ and $z_1=x_2$ for some $x_1, x_2 \in [1,\sqrt{2}]$.}

Then $||v_2-v_1|| \geq 2+\sqrt{2}$. On the other hand, since there is some other side $E_i$ of $S$ consisting of only two segments of lengths $l(x_i), l(x_{i+1})$ at most $1$ each, and since $E_1$ and $E_i$ are of equal lengths, we reach a contradiction.

\smallskip

\textbf{Case 2}:
\emph{$y_1=x_1$ and $z_1=l(x_2)$ for some $x_1,x_2 \in [1,\sqrt{2}]$.}

If some other side $E_i$ of $S$ contains two segments of lengths $l(x_3),l(x_4)$ of length at most $1$, then the length of $E_1$ is at least $1+\sqrt{2}$ and the length of $E_i$ is at most $2$, a contradiction. Thus, every side of $S$ contains exactly one segment of length $l(x)$ for some $x \in [1,\sqrt{2}]$. More specifically, for some $x_1,x_2,x_3,x_4 \in [1,\sqrt{2}]$, we have $y_i=x_i$ and $z_i=l(x_{i+1})$. In particular, in this case we have $x_1 + \sqrt{2}+l(x_2)=l(x_3) + x_2$. Note that $x_1+\sqrt{2}+l(x_2) \geq 1+\sqrt{2} \geq x_2 + l(x_3)$, with equality if and only if $x_1=1, x_2=\sqrt{2}$ and $x_3=1$. But then $x_3 +l(x_4) \leq 2$, contradicting the condition that $E_1$ and $E_3$ are of equal length.

\smallskip

\textbf{Case 3}:
\emph{$y_1=l(x_4)$ and $z_1=l(x_1)$ for some $x_1,x_4 \in [1,\sqrt{2}]$.}

In the following we set the value $\bar{x}= \frac{1}{2}\sqrt{2}\sqrt{\sqrt{13-8\sqrt{2}}+1}\approx1.072$.
At this point it is important to note that $\bar{x}$ is the unique solution of the equation $x+1=l(x)+\sqrt{2}$ because, on applying Lemma~\ref{lem:mainfor3},
the function $x-l(x) = x\sqrt{x^2-1}$ is strictly increasing on the entire interval $[1,\sqrt{2}]$.

By our conditions, there is a unique side of $S$, say $E'$, whose length is $x_i+x_j$ for some $x_i,x_j \in [1,\sqrt{2}]$. 

%\todo{a következő bekezdés több helyen módosítva.}

Assume first that $E'$ is $E_2$ or $E_4$, say $E'=E_2$. Then the lengths of $E_1, \ldots, E_4$ are $\sqrt{2}+l(x_1)+l(x_4)$, $x_1+x_2$, $l(x_2)+x_3$ and $l(x_3)+x_4$, respectively (all equal). By the construction exhibited in Figure \ref{fig:5bdoptimal}, which provides a lower bound on $S(5)$, $S$ has edge length at least $\bar{x}+1$, whereas all $x_i$ are at most $\sqrt{2}$. Thus, $l(x_2),l(x_3) \geq \bar{x}+1-\sqrt{2} = l(\bar{x})$. 
However, $l(x) > l(\bar{x})$ implies $x < \bar{x}$, therefore we have
$x_2, x_3 \leq \bar{x}$, from which we conclude $l(x_2)+x_3 \leq 1+\bar{x}$, implying the assertion.

Assume now that $E'=E_3$. Then the lengths of $E_1, \ldots, E_4$ are $\sqrt{2}+l(x_1)+l(x_4)$, $x_1+l(x_2)$, $x_2+x_3$ and $l(x_3)+x_4$, respectively.
By our assumptions, we have $l(x_2), l(x_3) \geq \bar{x}+1-\sqrt{2}$, implying that $x_2, x_3 \leq \bar{x}$. %Since $x_1+l(x_2) = x_2+x_3$, it follows that $x_1 \geq x_3$. It follows similarly that $x_4 \geq x_2$. Thus, $l(x_1)+l(x_4) \leq l(x_2)+l(x_3)$
Thus, denoting the edge length of $S$ by $a$, we obtain the equality
\begin{equation}\label{eq:impl}
l(a-l(x_2))+\sqrt{2}+l(a-l(a-x_2))-a = ||v_2-v_1||-a = 0.
\end{equation}
Consider the left-hand side in (\ref{eq:impl}) as the function $G(x_2,a)$. Assuming $x_3 \leq x_2$ without loss of generality, we obtain the inequalities $\frac{1+\bar{x}}{2} \leq x_2 \leq \bar{x}$ and $1+\bar{x} \leq a \leq 2x_2$. 
%\todo{eredetleg itt $1+2x_2$ volt.}
We show that within this domain, the function $a$, defined by the condition $G(x_2,a)=0$, is maximal if $x_2=\frac{1+\bar{x}}{2}$ or $x_2 = \bar{x}$. To show it, we observe that, by numeric computation, the second derivative $\frac{\mathop{d}^2 a}{\mathop{d} x_2^2}$ is positive, see Figure~\ref{fig:seconddiff}. We remark that this derivative can be computed by the expression
\[
\frac{\mathop{d}^2 a}{\mathop{d} x_2^2} = -\frac{G''_{x_2x_2} \left( G'_a \right)^2 - 2 G''_{x_2a} G'_{x_1} G'_a + G''_{aa} \left( G'_{x_2} \right)^2}{\left( G'_a \right)^3}.
\]

\begin{figure}[ht]
\begin{center}
\includegraphics[width=0.7\textwidth]{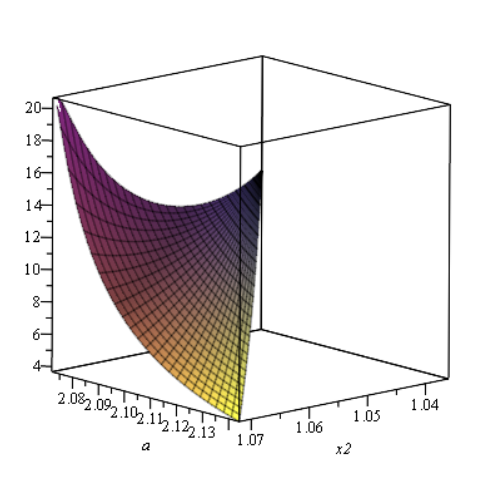}
\caption{The second derivative of $a$ as a function of $x_2$.}
\label{fig:seconddiff}
\end{center}
\end{figure}

Thus, it follows that $\frac{\mathop{d} a}{\mathop{d} x_2}$ is minimal if $x_2 = \frac{1+\bar{x}}{2}$. Using the symmetry of the configuration, it follows that $\left. \frac{\mathop{d} a}{\mathop{d} x_2} \right|_{x_2=(1+\bar{x})/2} =0$, implying that $a$, as a function of $x_2$, is strictly increasing.
This completes the proof of the theorem.

\section{Acknowledgments}
The authors express their gratitude to K. Bezdek for fruitful discussions related to this problem, and to J. Sgall for his comments on the next-to-final version.

\end{document}